\documentclass{amsart}
\usepackage[margin=1in]{geometry}
\usepackage{amsmath}
\usepackage{amssymb}
\usepackage{amsthm}
\usepackage{mathrsfs}
\usepackage{amscd}
\usepackage{amsfonts}
\usepackage{graphicx}%
\usepackage{fancyhdr}
\usepackage{soul}
\usepackage{xcolor}
\usepackage{hyperref}
\usepackage[semibold]{sourceserifpro}
\usepackage{newtxtext,newtxmath}
\usepackage{extpfeil}

\usepackage{thmtools} %for clever ref

\usepackage[capitalize]{cleveref}
\crefname{equation}{}{}
\crefname{enumi}{}{}
\creflabelformat{enumi}{(#2#1#3)}

\usepackage{tikz-cd}
\usepackage{tikz}
\usetikzlibrary{calc}
\usetikzlibrary{arrows}

\theoremstyle{plain} \numberwithin{equation}{section}
\newtheorem{theorem}{Theorem}[section]

\newtheorem{lemma}[theorem]{Lemma}
\newtheorem{proposition}[theorem]{Proposition}
\theoremstyle{definition}
\newtheorem{definition}[theorem]{Definition}

\newtheorem{remark}[theorem]{Remark}
\newtheorem{example}[theorem]{Example}
\newtheorem{question}{Question} \topmargin-2cm
\newtheorem*{theorem*}{Theorem}

\newcommand{\R}{\mathbb{R}}

\newcommand{\Z}{\mathbb{Z}}

\newcommand{\Q}{\mathbb{Q}}
\newcommand{\F}{\mathbb{F}}

\newcommand{\ip}[1]{\left\langle #1 \right\rangle}

\DeclareMathOperator{\conv}{conv}
\DeclareMathOperator{\spn}{span}

\usepackage[backend=biber,
        style=alphabetic,citestyle=alphabetic,
        maxnames=99,
        url=false,isbn=false,
        hyperref=true,backref=true]{biblatex}

\DefineBibliographyStrings{english}{
backrefpage = {$\uparrow$},
backrefpages = {$\uparrow$}
}
\title[On invariants of representations of Weyl groups]{On invariants of representations of Weyl groups associated with the cohomology of toric varieties}

\author{Tao Gong}
\address{Department of Mathematics, University of Western Ontario}
\email{tgong23@uwo.edu}

\subjclass[2020]{Primary 13A50; Secondary 14M25, 20F55, 52B20}

\keywords{Weyl group, cohomology of toric variety, permutohedron, polytopal algebra, invariant theory, root system}

\begin{document}
\begin{abstract}
  For a Weyl group $W$ and a $W$-permutohedron $P$, there are associated toric varieties $X_P$ and $X_{P/W_K}$ for any parabolic subgroup $W_K$ of $W$, since the quotient $P/W_K$ can be identified with a polytope inside $P$. 
  We construct an explicit algebra isomorphism between $H^*(X_{P/W_K};\Q)$ and $H^*(X_P;\Q)^{W_K}$. 
  We further generalize this isomorphism to intermediate lattices, to finite Coxeter groups, and to non-degenerate $W$-symmetric polytopes. 
  Our results give affirmative answers to two open questions of Horiguchi--Masuda--Shareshain--Song.
\end{abstract}
\maketitle

\section{Introduction}

In  a (real) Euclidean space $\left(V,\ip{-,-}\right)$ of dimension $n$, 
let $R$ be a reduced crystallographic root system with fixed simple system $S=\{\alpha_1,\alpha_2,\ldots,\alpha_n\}$.
 For a subset $K\subset S$, let $C_K$ denote the closed polyhedral cone consisting of vectors whose inner products with the roots in $K$ are non-negative.
 The Weyl group $W$ of $R$ is generated by reflections across the hyperplanes whose normal vectors are the simple roots in $S$.
 Let $P_{\lambda}$ be a \textbf{$W$-permutohedron}, that is, $P_{\lambda}=\conv\left(W(\lambda)\right)$ where $\lambda$ lies in the relative interior of $C_S$.
 Furthermore, we assume that  $\lambda$ lies in the root lattice. Then
  $P_{\lambda}$ is a simple integral polytope in the real span of the root lattice and defines a toric variety $X_{P_{\lambda}}$.
  The Weyl group $W$ acts on $P_{\lambda}$, and hence on the toric variety $X_{P_{\lambda}}$ and its rational cohomology ring $H^*(X_{P_{\lambda}};\Q)$. 
  This representation has been studied in \cite{procesiToricVarietyAssociated1990,stembridgePermutationRepresentationsWeyl1994,dolgachevCharacterFormulaRepresentation1994,lehrerRationalPointsCoxeter2008}, among other works.

 The \textbf{parabolic subgroup} $W_K$ of $W$ determined by the subset  $K$ of $S$ also acts on $P_{\lambda}$. 
 The quotient $P_{\lambda}/W_K$ can be identified with the polytope $P_{\lambda}\cap C_K$. 
 It is therefore associated with a toric variety $X_{P_{\lambda}/W_K}$.
  The relation between rational cohomology rings $H^*(X_{P_{\lambda}};\Q)$  and $H^*(X_{P_{\lambda}/W_K};\Q)$  has  attracted considerable attention. 
  Horiguchi--Masuda--Shareshain--Song \cite{horiguchiToricOrbifoldsAssociated2024} explicitly constructed an algebra isomorphism between $H^*(X_{P_{\lambda}/W_K};\Q)$ and $H^*(X_{P_{\lambda}};\Q)^{W_K}$ for $R$ of classical Lie types, and asked whether such an isomorphism exists for other root systems \cite[Questions 8.1, 8.2]{horiguchiToricOrbifoldsAssociated2024}.

These questions were previously addressed by Song \cite{songToricSurfacesReflection2022} for rank-two root systems, who also proved the isomorphism for $W$-symmetric polygons, and by Gui--Hu--Liu \cite{guiWeylGroupSymmetries2025} for arbitrary root systems in the case $W_K = W$. 
In this paper, we provide a complete affirmative answer to these questions.
One result of this paper is the following.
\begin{theorem}\label{coho ring iso}
  For any parabolic subgroup $W_K$, there is an explicit ring isomorphism
  $$H^*(X_{P_{\lambda}/W_K};\Q)\cong H^*(X_{P_{\lambda}};\Q)^{W_K}.$$
\end{theorem}

Our result in \cref{coho ring iso} relies on the fact that a toric variety can be constructed from a rational polytope (or, more generally, when the normal cones are rational) with respect to a lattice structure in $V$; see \cite[\S 5]{danilovGEOMETRYTORICVARIETIES1978}.
We also use the well-known description of the rational cohomology ring of a toric variety due to Danilov \cite{danilovGEOMETRYTORICVARIETIES1978} and Jurkiewicz \cite{jurkiewiczChowRingProjective1980}, which requires  the polytope to be $n$-dimensional and \textbf{simple} (i.e., each vertex lies in exactly $n$ facets).

We now put everything in a more general setting.
Let $M$ be a lattice lying between the root lattice and the weight lattice.
In terms of $M$,  we consider a full-dimensional  $W$-symmetric rational polytope $P_{\Lambda}$. Then  $P_{\Lambda}=\conv\left(W(\Lambda)\right)$ with $\Lambda$ a finite set of  vertices lying in $C_S$ and in the rational span of $M$.
We call $P_{\Lambda}$ \textbf{non-degenerate} if no vertex of $P_{\Lambda}$ lies on the boundary of $C_S$, and \textbf{degenerate} otherwise.
The quotient $P_{\Lambda}/W_K$ is identified with the polytope $P_{\Lambda}\cap C_K$ and is associated with the toric variety $X_{P_{\Lambda}/W_K}$.
One of the main results of this paper is the following.
\begin{theorem}\label{general coho ring iso}
  For a simple non-degenerate $W$-symmetric polytope $P_{\Lambda}$ and any parabolic subgroup $W_K$, there is an explicit ring isomorphism
  $$H^*(X_{P_{\Lambda}/W_K};\Q)\cong H^*(X_{P_{\Lambda}};\Q)^{W_K},$$
  which holds for any intermediate lattice $M$.
\end{theorem}

See \cref{eq:base map} for the explicit description of the isomorphism in \cref{general coho ring iso}, and \cref{prop:indep lattice} for its independence of the choice of $M$.
By taking $\Lambda$ to be a singleton and $M$ to be the root lattice, one obtains \cref{coho ring iso} as a special case of \cref{general coho ring iso}.
\begin{remark}
  In fact, $X_{P_{\Lambda}}/W_K$ and $X_{P_{\Lambda}/W_K}$ in \cref{general coho ring iso} are equivalent in certain respects.
  Blume \cite{blumeToricOrbifoldsAssociated2015} proved a variety isomorphism in the case where $R$ is of type $A$, $B$, or $C$, $\Lambda$ is a singleton, $M$ is the root lattice, and $W_K = W$.
  The author \cite{gongHomotopyTypesToric2024} established a homotopy equivalence.
  Crowley--G.--Simpson \cite{crowleyToricVarietiesModulo2024} proved a variety isomorphism that is also valid when $P_{\Lambda}$ is neither non-degenerate nor simple.
  However, none of these works described the explicit behavior of these equivalences on cohomology.
\end{remark}

More generally, we  consider a reduced root system $R$ that is not necessarily crystallographic. 
In this case, the Weyl group $W$ of $R$ may fail to preserve a lattice, and the $W$-symmetric polytope $P_{\Lambda}$ need not be rational; see, for example, \cite[Proposition 2.8]{humphreysReflectionGroupsCoxeter1992}.
However, the objects $S$, $K$, $W$, $W_K$, $P_{\Lambda}$, $C_K$, and $P_{\Lambda}/W_K$ can all be defined or chosen without reference to a lattice structure.
We then modify the construction in \cite[p.~248]{stembridgePermutationRepresentationsWeyl1994} to define a special graded polytopal algebra $\mathcal{A}(P_{\Lambda})$ over $\R$ (see \cref{section:polytopal algebra} for details).
 Roughly speaking, $\mathcal{A}(P_{\Lambda})$ is the quotient of the Stanley--Reisner ring of $P_{\Lambda}$ by the ideal induced by the normal vectors, where we take the normal vectors to lie on the unit sphere.

More generally, we can consider a reduced root system $R$ that is not necessarily crystallographic. 
In this case, the Weyl group $W$ of $R$ may fail to preserve a lattice, and the $W$-symmetric polytope $P_{\Lambda}$ need not be rational; see, for example, \cite[Proposition 2.8]{humphreysReflectionGroupsCoxeter1992}.
However, the objects $S$, $K$, $W$, $W_K$, $P_{\Lambda}$, $C_K$, and $P_{\Lambda}/W_K$ can all be defined or chosen without reference to a lattice structure.
We then modify the construction in \cite[p.~248]{stembridgePermutationRepresentationsWeyl1994} to define a special graded polytopal algebra $\mathcal{A}(P_{\Lambda})$ over $\R$ (see \cref{section:polytopal algebra} for details).
 Roughly speaking, $\mathcal{A}(P_{\Lambda})$ is the quotient of the Stanley--Reisner ring of $P_{\Lambda}$ by the ideal induced by the normal vectors, where we take the normal vectors to lie on the unit sphere.
This polytopal algebra $\mathcal{A}(P_{\Lambda})$ is intended as an analogue of the cohomology ring of a toric variety, although at present it may not have topological counterpart.

\begin{theorem}\label{general ana ring iso}
  For a simple non-degenerate $W$-symmetric polytope $P_{\Lambda}$ and any parabolic subgroup $W_K$, there is an explicit ring isomorphism
  $$\mathcal{A}(P_{\Lambda}/W_K)\cong \mathcal{A}(P_{\Lambda})^{W_K}.$$
\end{theorem}

See \cref{eq:general iso} for the explicit description of the isomorphism in \cref{general ana ring iso}.
By taking $R$ to be crystallographic and $M$ to be an intermediate lattice, one obtains \cref{general coho ring iso} as a special case of \cref{general ana ring iso}, with the help of \cref{prop:independent model}.

There is no straightforward generalization of the result in \cref{general ana ring iso} to degenerate $W$-symmetric polytopes, since their face structures are more difficult to describe.
However, this difficulty does not arise in dimension two. 
This generalization is proved in \cref{generalization to polygons}, and it recovers Song's result in \cite{songToricSurfacesReflection2022} when $R$ is crystallographic.

\begin{remark}
  One should distinguish the polytopal algebra from the \textbf{polytope algebra}  defined by McMullen \cite{mcmullenPolytopeAlgebra1989}, which is generated by all polytopes in $V$ under union and Minkowski sum, subject to certain relations.  This polytope algebra  is isomorphic to the completion of the polytopal algebras of all polytopes in $V$ with integral constant terms; see \cite{brionStructurePolytopeAlgebra1997}.
\end{remark}

Our constructions of the isomorphisms are motivated by the work in \cite{horiguchiToricOrbifoldsAssociated2024}.
While they treated the classical types separately, we provide a type-uniform proof.
\cite{guiWeylGroupSymmetries2025} also provided a uniform proof of their result, which instead depended on the face structure of $P/W$, and introduced a slightly different polytopal algebra.
Our work does not require special face structures and provides more details of the polytopal algebra.

In \cref{section:faces}, we rewiew the face structures of the involved polytopes and their Stanley--Reisner rings.
In \cref{section:variety coho}, we review the rational cohomology rings of toric varieties.
In \cref{section:proofs coho}, we prove \cref{general coho ring iso}.
In \cref{section:polytopal algebra}, we construct the  polytopal algebra $\mathcal{A}(P_{\Lambda})$ and prove \cref{general ana ring iso}.
In \cref{section:generalization}, we discuss questions concerning isomorphisms and representations for $\mathcal{A}(P_{\Lambda})$ when $P_{\Lambda}$ is degenerate.

{\subsection*{Acknowledgements} Thanks to Matthias Franz for much help, to Tao Gui for discussion and proofreading, to Kumar Shukla for discussion.}

\section{Face structures \& Stanley--Reisner rings}\label{section:faces}

In this section, we consider a reduced root system $R$ in the Euclidean space $V$, not necessarily crystallographic.
We describe the face structures of a full-dimensional $W$-symmetric polytope $P_{\Lambda}$ and its quotient polytope $P_{\Lambda}/W_K$.
The non-degenerate case when $\Lambda$ is a singleton (i.e., $P_{\Lambda}=P_{\lambda}$) was studied in \cite[\S\S 3,5]{gongHomotopyTypesToric2024}, which we follow and generalize here.
We also discuss the Stanley--Reisner rings of these polytopes.
Background on root systems and Weyl groups can be found in \cite[Part I]{humphreysReflectionGroupsCoxeter1992}.
When there is no ambiguity, we write $P$ instead of $P_{\Lambda}$.

Recall that $S$ consists of $n$ simple roots $\alpha_1,\alpha_2,\dots,\alpha_n$, and that $K$ is a subset of $S$.
We identify $S$ with the set $\{1,2,\dots,n\}$ and $K$ with the corresponding subset.
Each root $\alpha\in R$ defines a reflection $r_{\alpha}\in O(V)$ given by 
$$r_{\alpha}(x)=x-\frac{2\ip{x,\alpha}}{\ip{\alpha,\alpha}}\alpha.$$ 
The reflection $r_{\alpha}$ fixes a hyperplane pointwise, which we denote by $H_{\alpha}$.
For a simple root $\alpha_i \in S$, we write $r_i = r_{\alpha_i}$ and $H_i = H_{\alpha_i}$.
The parabolic subgroup $W_K$ is generated by the simple reflections $\{r_k:\, k \in K\}$.
In particular, $W_S = W$. We denote the identity element of $W$ by $e$.

Define the (closed) \textbf{fundamental chamber}
\begin{align*}
  C_K:=\left\{x\in V:\,\ip{x,\alpha_k}\ge 0,\,\, \forall k\in K\right\}.
\end{align*}
The fundamental chamber satisfies several important properties, summarized in the following theorem.
\begin{theorem}[{\cite[\S 1.12]{humphreysReflectionGroupsCoxeter1992}}]\label{thm:fundamental chamber}\phantom{}
  \begin{itemize}
    \item[$(1)$] The composite $C_K\hookrightarrow V\to V/W_K$ is a homeomorphism.
    \item[$(2)$] The stabilizer of $x \in C_K$ is generated by the simple reflections $r_k$ for which $r_k(x) = x$.
    \item[$(3)$] For $s \in W_K$ and $x \in C_K$, the difference $s(x) - x$ lies in the non-positive span of the simple roots corresponding to the minimal parabolic subgroup of $W_K$ containing $s$.
  \end{itemize}
\end{theorem}

Recall that $P=\conv\left(W(\Lambda)\right)$, where $\Lambda$ is a finite set of vertices lying in  $C_S$.  By \cref{thm:fundamental chamber}.$(1)$, we identify $P/W_K$ with the polytope $P\cap C_K$.
For a facet $F$ of $P$, $P/W_K$, or $C_S$, we denote by $\ell_F$ an outward-pointing normal vector of $F$.
In particular, for $i \in S$, the vector $\ell_{H_i}$ is a negative multiple of $\alpha_i$.

\begin{lemma}\label{lemma:intersection HF}
  Let $F$ be a face of $P$, and $r_{\alpha}\in W$ be the reflection for some $\alpha\in R$. The following conditions are equivalent:
  \begin{itemize}
    \item[$(1)$] The reflection $r_{\alpha}$ preserves $F$, that is, $r_{\alpha}(F)=F$.
    \item[$(2)$] The hyperplane $H_{\alpha}$ passes through the barycenter of $F$.
    \item[$(3)$] The hyperplane $H_{\alpha}$ passes through a relative interior point of $F$.
  \end{itemize}
 When $P$ is non-degenerate, there is an additional equivalent condition:
  \begin{itemize}
    \item[$(4)$] The face $F$ intersects the hyperplane $H_{\alpha}$.
  \end{itemize}
 Under this condition, we have  $\ip{\ell_{F'},\alpha}=0$ for any facet $F'$ of $P$ containing $F$.
\end{lemma}
\begin{proof}
  The implications  $(1)\Rightarrow(2)\Rightarrow(3)$ are immediate. We need to show that $(3)\Rightarrow (1)$: the relative interior point where $F$ and $H_{\alpha}$ meet is also contained in the face $r_{\alpha}(F)$. It follows that $r_{\alpha}(F) = F$.

  When $P$ is non-degenerate, it suffices to show $(4)\Rightarrow (1)$. Indeed, $F\cap r_{\alpha}(F)\cap H_{\alpha}\ne\emptyset$ implies that $r_{\alpha}(F)=F$ since no vertex lies in $H_{\alpha}$.

  The facet $F'$ intersects $H_i$. Then $r_{\alpha}(F')=F'$ and $\ell_{F'}=r_{\alpha}(\ell_{F'})=\ell_{F'}-\frac{2\ip{\ell_{F'},\alpha}}{\ip{\alpha,\alpha}}\alpha$.
  Thus, $\ip{\ell_{F'},\alpha}=0$.
  \qedhere
\end{proof}

\begin{remark}\label{remark:Vinberg}
  The equivalences $(1)\Leftrightarrow(2)\Leftrightarrow(3)$ in \cref{lemma:intersection HF} were stated for non-degenerate $P_{\lambda}$ in \cite[\S 3]{vinbergCERTAINCOMMUTATIVESUBALGEBRAS1991}.
\end{remark}

Let $\mathcal{F}_K$ denote the set of facets of $P$ that have barycenters lying in $C_K$. 
By \cref{thm:fundamental chamber}.(1), each facet of $P$ is $W_K$-conjugate to a unique facet in $\mathcal{F}_K$.
For a facet $F$ of $P$, let $W_F$ denote the stabilizer of $F$ under the $W_K$-action.
We then have the following result.

\begin{proposition}\label{prop:facets of polytopes}
  Define $\mathcal{F}:=\left\{(F,s):\, F\in\mathcal{F}_K,\, s\in{W_K}/{W_F}\right\}$.
  \begin{itemize}
    \item[$(1)$] There is a bijection from $\mathcal{F}_K\sqcup  K$ to the set of facets of $P/W_K$ given by 
     \begin{align*}
       F\in\mathcal{F}_K\mapsto F\cap C_K,\quad k\in K\mapsto H_k\cap P.
     \end{align*}
     \item[$(2)$] There is a bijection from $\mathcal{F}$ to the set of facets of $P$ given by 
     \begin{align*}
       (F,s)\mapsto s(F).
     \end{align*}
  \end{itemize}
\end{proposition}

We next prove a lemma that helps to clarify the face structure of $P$.

\begin{lemma}\label{lem:stabilizer}
  Fix $F\in\mathcal{F}_K$.
  \begin{itemize}
    \item[$(1)$] The stabilizer $W_F$ is generated by all  simple reflections $r_k\in W_K$ that fix the barycenter of $F$.
    \item[$(2)$] If $P$ is non-degenerate, then for every $s \in W_K$, the intersection $s(F) \cap F$ is either $F$ or empty.
  \end{itemize}
\end{lemma}
\begin{proof}
  $(1)$. If $r\in W_K$ satisfies $r(F)=F$, then $r$ fixes the barycenter of $F$. The conclusion therefore follows from \cref{thm:fundamental chamber}.$(2)$ and \cref{lemma:intersection HF}.

  $(2)$.  Suppose that $s(F) \cap F \neq \emptyset$. Then there exists  $t\in W_K$ such that $ts(F)\cap t(F)\cap C_K\ne\emptyset$. By \cref{lemma:intersection HF}, both $t s(F)$ and $t(F)$ lie in $\mathcal{F}_K$. 
It follows from \cref{prop:facets of polytopes}.$(2)$ that $ts(F)=t(F)=F$, and hence $s(F)=F$.
\qedhere
\end{proof}

As an immediate consequence of \cref{prop:facets of polytopes} and \cref{lem:stabilizer}, every face of $P$ can be written as an intersection of distinct facets
$s_1(F_1)\cap s_2(F_2)\cap \cdots\cap s_p(F_p),$
 where $F_i \in \mathcal{F}_K$ and $s_i \in W_K/W_{F_i}$. There is no repetition among those $F_i$ satisfying $F_i \cap \partial C_K = \emptyset$. If $P$ is non-degenerate, then there is no repetition among any of the $F_i$, and the intersection is $W_K$-conjugate to
 $F_1\cap F_2\cap \cdots\cap F_p.$

A polytope is called \textbf{flag} if any collection of its facets has non-empty intersection whenever every pair of facets in the collection has non-empty intersection.

\begin{proposition}\label{lem:simple flag}
  Suppose that $P$ is non-degenerate.
  \begin{itemize}
    \item[$(1)$] The polytope $P$ is simple if and only if $P/W_K$ is simple.
    \item[$(2)$] If $P$ is flag, then so is $P/W_K$.
  \end{itemize}
\end{proposition}
\begin{proof}
  $(1)$. Suppose that $P$ is simple and a vertex of $P/W_K$ is contained in facets determined, via \cref{prop:facets of polytopes}.$(1)$, by some maximal subsets $K'\subset \mathcal{F}_K$ and  $K''\subset K$. Then  $|K'|+|K''|\ge n$, and  
  \begin{align*}
    n=\dim_{\R}\left(\spn_{\R}\left(\ell_{F_{i}}: F_i\in K'\right)\right)+\dim_{\R}\left(\spn_{\R}\left(\ell_{H_{j}}: j\in K''\right)\right)=|K'|+|K''|,
  \end{align*}
  where the first equality holds because the two spans are orthogonal complements of each other, by \cref{lemma:intersection HF}. Hence $P/W_K$ is simple.

  Conversely, suppose that $P/W_K$ is simple. Each vertex of $P$ is $W_K$-conjugate to a vertex of $P/W_K$, which is also the intersection of exactly $n$ facets from $\mathcal{F}_K$.
Hence $P$ is simple.

  $(2)$. Suppose that $F_{1}, F_{2},\ldots, F_{p}$ and $H_{k_1}, H_{k_2},\ldots, H_{k_q}$ support facets of $P/W_K$  such that every two of them have non-empty intersection.
  Then $F:=F_{1}\cap F_{2}\cap\cdots\cap F_{p}$ is a non-empty face of $P$. 
  Moreover, each $H_{k_j}$, for $1\leq j\leq q$, passes through the  barycenter of $F$.
  Therefore, 
  $F\cap H_{k_1}\cap\cdots\cap H_{k_q}\ne \emptyset$, and hence $P/W_K$ is flag.
  \qedhere
\end{proof}

\begin{remark}
  When $\Lambda$ is a singleton (i.e., $P_{\Lambda} = P_{\lambda}$), the results of \cref{lem:simple flag} were obtained combinatorially in \cite{horiguchiToricOrbifoldsAssociated2024}.
See also \cite{burrullStronglyDominantWeight2024} for the case $W_K = W$.
\end{remark}

  \begin{example}\label{example: I(5) and A_2}
    When $R$ is of type $I_2(5)$, the group $W$ is the dihedral group of order $10$; if $\lambda$ is chosen in $H_2$, then $P_{\lambda}$ is a pentagon.
When $R$ is of type $A_2$, the group $W$ is the dihedral group of order $6$;
if $\lambda$ is chosen in $C_S \setminus \partial C_S$, then $P_{\lambda}$ is a hexagon.
The polytopes $P_{\lambda}$ and $P_{\lambda}/W$ are shown in \cref{figure:polytopes}. The $W$-actions on the two permutohedra are described as follows:
\begin{itemize}
\item 
For type $I_2(5)$ with $\lambda \in H_2$, 
the reflection $r_1$ maps $E_i$ to $E_j$ if $i + j \equiv 2 \pmod{5}$, 
and $r_2$ maps $E_i$ to $E_j$ if $i + j \equiv 1 \pmod{5}$.

\item 
For type $A_2$ with $\lambda \in C_S \setminus \partial C_S$, 
the reflection $r_1$ maps $E_i$ to $E_j$ if $i + j \equiv 2 \pmod{6}$, 
and $r_2$ maps $E_i$ to $E_j$ if $i + j \equiv 0 \pmod{6}$.
\end{itemize}
    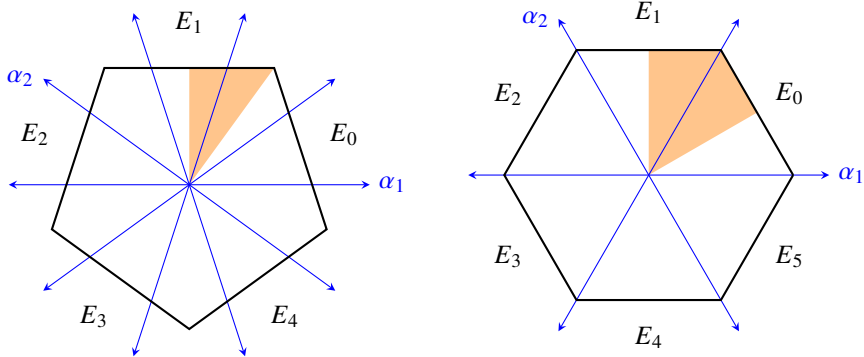
\begin{figure}[h]
      
      \centering

      \begin{tabular}{ccc}
        
      \begin{tikzpicture}[scale=0.8]%0
       
      \newdimen\r
      \r=3cm

      \fill[orange!50,opacity=0.9] (0:0) -- (54:{0.8*\r}) -- (90:{cos(36)*0.8*\r}) -- cycle;

      \foreach \x in {36,72,...,360} {
        \draw[blue,->,>=stealth] (0:0) -- (\x:\r);
      }
      \node[right,blue] at (0:\r) {$\alpha_1$};
      \node[left,blue] at (144:{\r}) {$\alpha_2$};

      \foreach \x in {0,1,...,4} {
        \node at ({18+72*\x}:{0.9*\r}) {$E_{\x}$};
      }

      \draw[thick] (54:{0.8*\r})-- (126:{0.8*\r})-- (198:{0.8*\r})--(270:{0.8*\r})--(342:{0.8*\r})--cycle;

      \end{tikzpicture}

      & \quad  &

      \begin{tikzpicture}[scale=0.8]%0
       
      \newdimen\r
      \r=3cm
      
      %draw chamber
      \fill[orange!50,opacity=0.9] (0:0) -- (30:{cos(30)*0.8*\r}) -- (60:{0.8*\r}) -- (90:{cos(30)*0.8*\r}) -- cycle;
       
      %draw roots
      \foreach \x in {60,120,...,360} {
        \draw[blue,->,>=stealth] (0:0) -- (\x:\r);
      }

      %label roots
      \node[right,blue] at (0:\r) {$\alpha_1$};
      \node[left,blue] at (120:{\r}) {$\alpha_2$};
      
      %draw polytope
      \draw[thick] (0:{0.8*\r}) \foreach \x in {0,60,...,360} { -- (\x:{0.8*\r}) };
      
      %\label edges

      \foreach \x in {0,1,...,5} {
        \node at ({30+60*\x}:{0.9*\r}) {$E_{\x}$};
      }

      \end{tikzpicture}

      \end{tabular}
      \caption{Degenerate polytope $P_{\lambda}$ of type $I_2(5)$ (left) and non-degenerate polytope $P_{\lambda}$ of type $A_2$ (right).
      The blue arrows indicate roots, the orange area indicates $P_{\lambda}\cap C_S$. Each $E_i$ denotes the edge.}\label{figure:polytopes}
    \end{figure}
  \end{example}

\begin{definition}
Let $Q$ be a polytope and $\F$ a field. The \textbf{Stanley--Reisner ring} of $Q$ over $\F$ is defined as
$$
\mathcal{SR}(Q)
  := \frac{\F[X_F :\, \mbox{$F$ is a facet of $Q$}]}{I},
$$
where the ideal $I$ is generated by all monomials
$X_{F_1} X_{F_2} \cdots X_{F_p}$
such that
$F_1 \cap F_2 \cap \cdots \cap F_p = \emptyset.$
Consequently, if $Q$ is flag, then $I$ is generated by the quadratic monomials
$X_{F_1} X_{F_2}$
with $F_1 \cap F_2 = \emptyset$.
\end{definition}

 By \cref{prop:facets of polytopes}, the Stanley--Reisner ring of $P$ is 
    $$\mathcal{SR}(P)=\frac{\F[X_{F,s}:\, (F,s)\in \mathcal{F}]}{I_P}$$
    where the ideal $I_P$ is generated by all  monomials $X_{F_1,s_1}X_{F_2,s_2}\cdots X_{F_p,s_p}$ such that $s_{1}(F_{1})\cap s_{2}(F_{2})\cap\cdots\cap s_{p}(F_{p})=\emptyset$. Moreover, $\mathcal{SR}(P)$ carries a natural $W_K$-action induced by
\begin{align}
  W_K\times \F[X_{F,s}:\, (F,s)\in \mathcal{F}] &\to \F[X_{F,s}:\, (F,s)\in \mathcal{F}],\label{eq:W-action}\\
  \left(t,X_{F,s}\right)&\mapsto X_{F,ts}.\nonumber 
\end{align}
The Stanley--Reisner ring of $P/W_K$ is 
$$\mathcal{SR}(P/W_K)=\frac{\F[X_{F,e}, Y_k:\, F\in\mathcal{F}_K, k\in K]}{I_K}$$
where the ideal $I_K$ is generated by all  monomials $X_{F_1,e}X_{F_2,e}\cdots X_{F_p,e}Y_{k_1}Y_{k_2}\cdots Y_{k_q}$ such that $F_{1}\cap F_{2}\cap\cdots\cap F_{p}\cap H_{k_1}\cap H_{k_2}\cap\cdots\cap H_{k_q}=\emptyset$.

\begin{lemma}\label{lem:I_K generators}
  Suppose that $P$ is non-degenerate. In $\mathcal{SR}(P/W_K)$, the ideal $I_K$ is generated by monomials of the following two types:
  \begin{itemize}
    \item monomials $X_{F_1,e}X_{F_2,e}\cdots X_{F_p,e}$ with $F_{1}\cap F_{2}\cap\cdots\cap F_{p}=\emptyset$.
    \item monomials $X_{F,e}Y_k$ with $F\cap H_k=\emptyset$.
  \end{itemize}
\end{lemma}
  \begin{proof}
    The monomial $Y_{k_1}Y_{k_2}\cdots Y_{k_q}$ is not in $I_K$ since $0\in P\cap H_{k_1}\cap H_{k_2}\cap\cdots\cap H_{k_q}$. 
    It therefore suffices to show that any monomial of the form
     $X_{F_1,e}X_{F_2,e}\cdots X_{F_p,e}Y_{k_1}Y_{k_2}\cdots Y_{k_q}$ in $I_K$ with $p+q \ge 2$ is generated by monomials of the two types above.
     Equivalently, we must show that
$$
F_1 \cap \cdots \cap F_p
\cap H_{k_1} \cap \cdots \cap H_{k_q}
\neq \emptyset
$$
whenever $F_1 \cap \cdots \cap F_p \neq \emptyset$ and each $F_{i}\cap H_{k_j}\ne\emptyset$ for $1\leq i\leq p$ and $1\leq j\leq q$.
By \cref{lemma:intersection HF}, the reflection $r_{k_j}$  preserves $F_{i}$, and hence preserves $F_{1}\cap F_{2}\cap\cdots\cap F_{p}$. Then $H_{k_1}\cap H_{k_2}\cap\cdots\cap H_{k_q}$ passes through the barycenter of $F_{1}\cap F_{2}\cap\cdots\cap F_{p}$.
\qedhere
  \end{proof}

\section{Cohomology rings with actions}\label{section:variety coho}

In this section, we continue the discussion of \cref{section:faces}. Additionally, we assume that the root system $R$ is crystallographic and that the polytope $P_{\Lambda}$ is simple and non-degenerate, with $\Lambda$ lying in the rational span of $M$.
Recall that $M$ is a lattice lying between the root lattice and the weight lattice. We describe the rational cohomology rings of the associated toric varieties and show that the cohomology ring is independent of the choice of $M$.
The ground field is always $\Q$ unless specified otherwise.
We write $H^*(-)$ for rational cohomology.

 By definition, the polytope $P=\conv\left(W(\Lambda)\right)$ is rational with respect to $M$. Then the quotient polytope $P/W_K$ is simple (by \cref{lem:simple flag}) and rational with respect to $M$.

 For a facet $F$ of $P$, of $P/W_K$, or of $C_S$, we additionally require the normal vector $\ell_F$ of $F$ to be primitive with respect to the lattice dual to $M$ via the inner product.
 The normal vectors of $P$ and of $P/W_K$ form complete simplicial fans, and hence define toric varieties $X_P$ and $X_{P/W_K}$, respectively; see \cite[\S 5]{danilovGEOMETRYTORICVARIETIES1978}.

\begin{theorem}[c.f. {\cite[\S 10.9]{danilovGEOMETRYTORICVARIETIES1978}}]\label{thm:danilov cohomology}
Let $Q$ be a full-dimensional simple rational polytope with respect to a lattice $L$ of rank $n$ in $V$. The rational cohomology ring of the corresponding toric variety $X_Q$ is 
\begin{align*}
  H^*(X_Q)=\frac{\mathcal{SR}(Q)}{J}=\frac{\Q[X_{F}: \mbox{$F$ is a facet of $Q$\,}]}{I+J},
\end{align*}
where each $X_F$ has degree $2$, and the ideal $J$ of $\mathcal{SR}(Q)$ is generated by linear forms
  $$\sum_{F}\ip{u,\ell_F}X_F\,,\quad\mbox 
  {where $F$ ranges over all facets of $Q$,}$$
for all $u\in L$.
\end{theorem}

The following result shows that the rational cohomology is, to some extent, independent of the choice of lattice. 

\begin{proposition}\label{prop:indep lattice}
  Under the condition of \cref{thm:danilov cohomology}, 
  let $L'$ be a sub-lattice of $L$ of finite  index. 
  Let $X_Q$ and $X'_Q$ be the toric varieties defined with respect to the lattices $L$ and $L'$, respectively.
  Then there is an isomorphism between the rational cohomology rings $H^*(X_Q)$ and $H^*(X'_Q)$.
\end{proposition}
\begin{proof}
  Suppose that $L$ has a basis $e_1,e_2,\ldots,e_n$. Then the ideal $J$ of $\mathcal{SR}(Q)$ is generated by $n$ linear forms, for $1\leq i\leq n$,
  \begin{align*}
    \eta_i:= \sum_{F}\ip{e_i,\ell_F}X_F,\quad\mbox 
  {where $F$ ranges over all facets of $Q$.}
  \end{align*}
  We need to show that the two algebras below, obtained by adjusting the generators $\eta_i$ of $J$, are isomorphic to $H^*(X_Q)$:
  \begin{itemize}
    \item the algebra $\mathcal{H}'$ obtained from $H^*(X_Q)$  by changing each $e_i$ into $b_i\cdot e_i$ with $b_i\in\Q\setminus 0$,
    \item the algebra $\mathcal{H}''$ obtained from $\mathcal{H}'$  by changing each $\ell_F$ into $c_F\cdot \ell_F$ with $c_F\in\Q\setminus 0$.
  \end{itemize}
  We have $\mathcal{H}'=H^*(X_Q)$ since this operation does not change the ideal $J$.
   There is an isomorphism from $\mathcal{H}''$ to $\mathcal{H}'$ given by sending each $X_F$ to $X_F/c_F$. 
   Clearly, $H^*(X'_Q)$ coincides with an algebra of the form $\mathcal{H}''$.
   \qedhere
\end{proof}

We can now describe the rational cohomology rings of $X_P$ and $X_{P/W_K}$ without specifying the lattice $M$ explicitly.

\begin{theorem}\label{thm:cohomology of variety}\phantom{}
  \begin{itemize}
    \item[$(1)$] The cohomology ring $H^*(X_{P})$ is isomorphic to 
  $$ \frac{\mathcal{SR}(P)}{J_P}=\frac{\Q[X_{F,s}: (F,s)\in \mathcal{F}]}{I_P+J_P},
  $$
  where each $X_{F,s}$ has degree $2$, and the ideal $J_P$ is generated by linear forms
  $$\sum_{(F,s)\in\mathcal{F}}\ip{u,\ell_{s(F)}}X_{F,s}\qquad \mbox{for all}\,\, u\in M.$$
  
  \item[$(2)$] The cohomology ring $H^*(X_{P/W_K})$ is isomorphic to 
  $$ \frac{\mathcal{SR}(P/W_K)}{J_K}=\frac{\Q[X_{F,e}, Y_k:\, F\in\mathcal{F}_K,\, k\in K]}{I_K+J_K},
  $$
  where each $Y_k$ is of degree 2, and the ideal $J_K$ is generated by linear forms 
  $$\sum_{F\in\mathcal{F}_K}\ip{u,\ell_{F}}X_{F,e}+\sum_{k\in K}\ip{u,\ell_{H_k}}Y_k\qquad \mbox{for all}\,\, u\in M.$$
  \end{itemize}
\end{theorem}

The $W_K$-action on $\mathcal{SR}(P)$ defined in \cref{eq:W-action} induces a $W_K$-action on $H^*(X_P)$.
The canonical quotient map $\mathcal{SR}(P)\to H^*(X_P)$ is therefore $W_K$-equivariant.

\begin{example}\label{example: coho ring A_2}
  We continue from \cref{example: I(5) and A_2} to describe $H^*(X_{P_{\lambda}})$ and $H^*(X_{P_{\lambda}/W})$  when $R$ is of type $A_2$ and $M$ is the root lattice. Following \cref{figure:polytopes}, we have that $\mathcal{F}_S=\{E_0,E_1\}$ and
  $$W(E_0)=\{E_0,E_2=r_1(E_0),E_4=r_2r_1(E_0)\},\quad W(E_1)=\{E_1,E_3=r_1r_2(E_1),E_5=r_2(E_1)\}.$$
Moreover, 
  $$\ip{\ell_{E_0},\alpha_1}=1,\quad \ip{\ell_{E_0},\alpha_2}=0,\quad\ip{\ell_{E_1},\alpha_1}=0,\quad \ip{\ell_{E_1},\alpha_2}=1,$$
  $$\ell_{H_1}=-2\ell_{E_0}+\ell_{E_1},\quad \ell_{H_2}=\ell_{E_0}-2\ell_{E_1}.$$
  The cohomology ring $H^*(X_{P_\lambda})$ is 
  $$ \frac{\mathcal{SR}(P_\lambda)}{J_{P_\lambda}}=\frac{\Q[X_{E_{i}}: 0\leq i\leq 5]}{I_{P_\lambda}+J_{P_\lambda}},
  $$
  where $I_{P_\lambda}=\ip{X_{E_i}X_{E_j}:i-j\not\equiv \pm 1\pmod{6}}$, and $J_{P_\lambda}$ is generated by two linear forms
  \begin{align*}
    X_{E_0}-X_{E_2}-X_{E_3}+X_{E_5},\quad X_{E_1}+X_{E_2}-X_{E_4}-X_{E_5}.
  \end{align*}
 The cohomology ring $H^*(X_{P_{\lambda}/W})$ is 
  $$ \frac{\mathcal{SR}(P_{\lambda}/W)}{J_S}=\frac{\Q[X_{E_0}, X_{E_1},Y_1, Y_2]}{I_S+J_S},
  $$
  where $I_S=\ip{X_{E_0}Y_1,\, X_{E_1}Y_2}$, and $J_{S}$ is generated by two linear forms
  \begin{align*}
    X_{E_0}-2Y_1+Y_2,\quad X_{E_1}+Y_1-2Y_2.
  \end{align*}
\end{example}

\section{Construction of cohomological isomorphism}\label{section:proofs coho}
In this section, we continue the discussion of \cref{section:variety coho}. We construct and prove the isomorphism stated in \cref{general coho ring iso}.

For a facet $F\in\mathcal{F}_K$ and $s\in W_K$, we have $s(\ell_{F})=\ell_{s(F)}$ since $W$ preserves the lattice $M$ and its dual. 
By  \cref{thm:fundamental chamber}.$(3)$, we have
\begin{align}
  s(\ell_{F})-\ell_{F}=\sum_{k\in K} C_{F,s,k}\ell_{H_k},\label{eq:coefficients of s}
\end{align}
 for some non-negative rational coefficients $C_{F,s,k}$. Moreover, $C_{F,s,k}>0$ only if the simple reflection $r_k$ lies in the minimal parabolic subgroup containing $s$.

Define the homomorphism
\begin{align}
  \Phi: \Q[X_{F,e}, Y_k: \, F\in\mathcal{F}_K,\, k\in K] \longrightarrow H^*(X_{P})\label{eq:topmap}
\end{align}
$$
  \Phi(X_{F,e})=\sum_{s\in{W_K}/{W_F}}X_{F,s}\quad\text{and}\quad \Phi(Y_{k})=\sum_{(F,s)\in\mathcal{F}} C_{F,s,k}X_{F,s}\, .
$$
This construction is motivated by that in \cite[Equation (5.3)]{horiguchiToricOrbifoldsAssociated2024}. We will show that $\Phi$ descends to the quotient and induces a homomorphism
$H^*(X_{P/W_K}) \to H^*(X_P)^{W_K}$.

\begin{lemma}
  The image of $\Phi$ lies in $H^*(X_{P})^{W_K}$.
\end{lemma}
\begin{proof}
  It suffices to show that $\Phi(Y_k)\in H^*(X_P)^{W_K}$, since $\Phi(X_{F,e})$ is manifestly $W_K$-invariant.
  In $H^*(X_{P})$, we have 
  $$
  \sum_{(F,s)\in\mathcal{F}}\ip{u,\ell_{s(F)}}X_{F,s}=0.
  $$
  Choose $u\in M$ to satisfy that $\ip{u,\ell_{H_k}}\ne 0$ while $\ip{u,\ell_{H_{k'}}}=0$ for $k'\ne k$. Via \cref{eq:coefficients of s}, we have 
  \begin{align*}
    \sum_{(F,s)\in\mathcal{F}}\ip{u,\ell_{F}}X_{F,s}+\sum_{(F,s)\in\mathcal{F}}\Bigl\langle {u,\sum_{j\in K} C_{F,s,j}\ell_{H_j}}X_{F,s}\Bigr\rangle&=0,\\
    \sum_{(F,s)\in\mathcal{F}}\ip{u,\ell_{F}}X_{F,s}+\ip{u,\ell_{H_k}}\sum_{(F,s)\in\mathcal{F}}C_{F,s,k}X_{F,s}&=0.
  \end{align*}
  Clearly $\sum_{\mathcal{F}}\ip{u,\ell_{F}}X_{F,s}$ is $W_K$-invariant, then so is $\sum_{\mathcal{F}}C_{F,s,k}X_{F,s}$.  
  \qedhere
\end{proof}

\begin{lemma}\label{lem:kernel I_K}
  The kernel of $\Phi$ contains $I_K$.
\end{lemma}
\begin{proof}
  First we show that $X_{F_1,e}X_{F_2,e}\cdots X_{F_p,e}$ is mapped to 0 under $\Phi$ when $F_{1}\cap F_{2}\cap\cdots\cap F_{p}=\emptyset$. 
  We have
  \begin{align*}
    \Phi\left(X_{F_1,e}X_{F_2,e}\cdots X_{F_p,e}\right)&=\biggl(\sum_{s_1\in {W_K}/{W_{F_1}}}X_{F_1,s_1}\biggr)\biggl(\sum_{s_2\in {W_K}/{W_{F_2}}}X_{F_2,s_2}\biggr)\cdots\biggl(\sum_{s_p\in {W_K}/{W_{F_p}}}X_{F_p,s_p}\biggr)\\
    &=\sum_{s_1,\ldots,s_p} X_{F_1,s_1}X_{F_2,s_2}\cdots X_{F_p,s_p}\,.
  \end{align*}
Each summand $X_{F_1,s_1}X_{F_2,s_2}\cdots X_{F_p,s_p}$ is $W_K$-conjugate to the item $X_{F_1,e}X_{F_2,e}\cdots X_{F_p,e}$ in $H^*(X_P)$, because the corresponding faces $s_1(F_{1})\cap s_2(F_{2})\cap\cdots\cap s_{p}(F_{p})$ and $F_{1}\cap F_{2}\cap\cdots\cap F_{p}$ of $P$ are $W_K$-conjugate.

  Next we show that $X_{F,e}Y_{k}$ is mapped to 0 under $\Phi$ when $F\cap H_{k}=\emptyset$. We have 
  \begin{align*}
    \Phi\left(X_{F,e}Y_{k}\right)&=\sum_{s\in {W_K}/{W_{F}}}X_{F,s}\Phi(Y_{k})=\sum_{s\in {W_K}/{W_{F}}}s\left(X_{F,e}\Phi(Y_{k})\right)\\
    &=\sum_{s\in {W_K}/{W_{F}}}s\biggl(\sum_{(F',t)\in\mathcal{F}}  C_{F',t,k}X_{F,e}X_{F',t}\biggr). 
  \end{align*}
  It remains to show that if  $X_{F,e}X_{F',t}\ne 0$ then $C_{F',t,k}=0$. Since $F\cap t(F')\ne \emptyset$, there exists $w\in W_K$ such that $w(F)\cap wt(F')=F\cap F'$, hence $w\in W_{F}$ and $wt \in W_{F'}$. We have, by \cref{thm:fundamental chamber}.$(3)$ and \cref{lem:stabilizer}.$(1)$, that
  \begin{align*}
    t(\ell_{F'})-\ell_{F'}=w^{-1}(\ell_{F'})-\ell_{F'}\in \spn_{\Q}\left\{\ell_{H_j}: H_j\cap F\ne\emptyset\right\},
  \end{align*} 
 Since $\ell_{H_k}$ does not belong to this span, we conclude that $C_{F',t,k}=0$.

We can omit all other cases, due to \cref{lem:I_K generators}. 
\qedhere
\end{proof}

\begin{lemma}
  The kernel of $\Phi$ contains $J_K$.
\end{lemma}
\begin{proof}
  We have, for any $u\in M$,
  \begin{align*}
    &\Phi\biggl(\sum_{F\in\mathcal{F}_K}\ip{u,\ell_{F}}X_{F,e}+\sum_{k\in K}\ip{u,\ell_{H_k}}Y_k\biggr)\\
    &=\sum_{F\in\mathcal{F}_K}\,\sum_{s\in{W_K}/{W_F}}\ip{u,\ell_{F}}X_{F,s}+\sum_{k\in K}\sum_{(F,s)\in\mathcal{F}}\ip{u,\ell_{H_k}}C_{F,s,k}X_{F,s}\\
    &=\sum_{(F,s)\in\mathcal{F}}\ip{u,s(\ell_{F})}X_{F,s}\in J_P,
  \end{align*}
  where the last equality follows from \cref{eq:coefficients of s}. This completes the proof.
\qedhere
\end{proof}

Therefore the homomorphism $\Phi$ in \cref{eq:topmap} induces a homomorphism 
\begin{align}\label{eq:base map}
  \phi: H^*({X_{P/W_K}})\to H^*(X_{P})^{W_K}.
\end{align}
We will show that $\phi$ is an isomorphism.

\begin{lemma}\label{lem:equiv surjection} Let $G$ be a finite group, and $A,B$ be $\Q[G]$-modules. If a module homomorphism $f:A\to B$ is surjective, then the induced homomorphism between invariants $\bar{f}:A^G\to B^G$ is also surjective. 
\end{lemma} 
\begin{proof} 
  For $b\in B^G$, there exists some $a\in A$ such that $f(a)=b$. Then $\bar{f}\left(\frac{1}{|G|}\sum_{g\in G}g\cdot a\right)=b$. 
  \qedhere 
\end{proof}

\begin{lemma}\label{lem:image}
  Recall that for $F\in\mathcal{F}_K$, we have
  $\Phi(X_{F,e})=\sum_{s\in W_K/W_F} X_{F,s}$,
  and similarly for $\phi(X_{F,e})$.
  \begin{itemize}
    \item[$(1)$] The algebra $\mathcal{SR}(P)^{W_K}$ is generated by $\Phi(X_{F,e})$ for all $F\in\mathcal{F}_K$.
    \item[$(2)$] The algebra $H^*(X_{P})^{W_K}$ is generated by $\phi(X_{F,e})$ for all $F\in\mathcal{F}_K$.
  \end{itemize}
\end{lemma}
\begin{proof}
  It suffices to show $(1)$, since the  quotient homomorphism $\mathcal{SR}(P)\to H^*(X_{P})$ induces a surjection $\mathcal{SR}(P)^{W_K}\to H^*(X_{P})^{W_K}$ by \cref{lem:equiv surjection}.

For a nonzero element in $\mathcal{SR}(P)^{W_K}$, we can choose its representative $f$ in 
$\Q[X_{F,s}: (F,s)\in\mathcal{F}]$ which is also $W_K$-invariant by \cref{lem:equiv surjection}. 
One can assume 
\begin{align*} 
  f= c\cdot X_{F_1,s_1}^{n_1}X_{F_2,s_2}^{n_2}\cdots X_{F_p,s_p}^{n_p}+\text{other terms $A$}, 
\end{align*} 
where $c\neq 0$ and $s_1(F_{1})\cap s_2(F_{2})\cap\cdots \cap s_p(F_{p})\ne \emptyset$. 
There exists  $\tilde{s}\in W_K$ such that 
$\tilde{s}\left(s_1(F_{1})\cap\cdots \cap s_p(F_{p})\right)=F_{1}\cap\cdots\cap F_{p}$. 
The stablizer of the face $F_{1}\cap\cdots\cap F_{p}$ is the parabolic subgroup $W':=W_{F_1}\cap\cdots\cap W_{F_p}$.
Using $W_K$-invariance, we may therefore assume 
\begin{align*} 
  f= c\cdot\biggl(\sum_{s\in {W_K}/{W'}} X_{F_1,s}^{n_1}X_{F_2,s}^{n_2}\cdots X_{F_p,s}^{n_p}\biggr)+A. 
\end{align*}
A non-empty face of the form $s(F_1\cap \cdots \cap F_p)$ can be written uniquely as
$t_1(F_{1})\cap\cdots \cap t_p(F_{p})$ with $t_i\in W_K/W_{F_i}$.
Hence
\begin{align*} 
  f&= c\cdot\biggl(\sum_{t_1,\ldots,t_{p}} X_{F_1,t_1}^{n_1}X_{F_2,t_2}^{n_2}\cdots X_{F_p,t_p}^{n_p}\biggr)+A+\mbox{other terms in $I_P$},\\ &=c\cdot\biggl(\sum_{t_1} X_{F_1,t_1}\biggr)^{n_1}\biggl(\sum_{t_2} X_{F_2,t_2}\biggr)^{n_2}\cdots \biggl(\sum_{t_p} X_{F_p,t_p}\biggr)^{n_p}+A+\mbox{other terms in $I_P$}, 
\end{align*} 
where $(t_1,\ldots,t_p)$ ranges over $\frac{W_K}{W_{F_1}}\times\cdots\times\frac{W_K}{W_{F_p}}$.
 The second equality follows from \cref{lem:stabilizer}.$(2)$ which implies that 
 $$X_{F_1,t_1}X_{F_1,w_1}\notin I\Leftrightarrow t_1(F_{1})\cap w_1(F_{1})\ne\emptyset\Leftrightarrow t_1=w_1\in {W_K}/{W_{F_1}}.$$ 
We may then proceed to the remaining term $A$ and argue by induction on the number of monomials.
This completes the proof.
\end{proof}

For any field $\F$, a finite-dimensional graded  $\F$-algebra $R=\bigoplus_{i=0}^dR_i$ with $R_0\cong R_d\cong\F$ is called a \textbf{Poincar\'{e}  duality algebra} (PDA) if the multiplication $R_i\times R_{d-i}\to R_d$ is non-degenerate for $0\leq i\leq d$. The rational cohomology ring of a toric orbifold is a PDA over $\Q$; see \cite[\S 14]{danilovGEOMETRYTORICVARIETIES1978}.

\begin{lemma}[\cite{abeCohomologyRingsRegular2019}, Lemma 10.5]\label{lem:sur is iso}
  Let $f:{A}=\bigoplus_{i=0}^dA_i\to {B}=\bigoplus_{i=0}^dB_i$ be a surjective homomorphism between graded $\F$-algebras with $A_d, B_d$ nonzero. If ${A}$ is a PDA and $f|_{A_d}:A_d\to B_d$ is an isomorphism, then $f$ is an isomorphism.
\end{lemma}
\begin{proof}
  The composite 
  $$\Bigl(A_i\times A_{d-i}\to A_d\xrightarrow[\cong]{f|_{A_d}}B_d\Bigr)=\Bigl(A_i\times A_{d-i}\xrightarrow{f|_{A_i}\times f|_{A_{d-i}}} B_i\times B_{d-i}\to B_d\Bigr)$$
  is non-degenerate. Then each $f|_{A_i}$ has to be injective, and $f$ is an isomorphism.
  \qedhere
\end{proof}

\begin{proof}[\textbf{Proof of \cref{general coho ring iso}}]
We have obtained the  homomorphism $\phi: H^*({X_{P/W_K}})\to H^*(X_{P})^{W_K}$ in \cref{eq:base map}. By \cref{lem:image}, $\phi$ is surjective.
Since the $W_K$-action on $X_P$ preserves orientation, we have
$(H^{2n}(X_{P}))^{W_K}=\Q$. 
Therefore \cref{lem:sur is iso} implies that the surjection $\phi$ is an isomorphism.
\qedhere
\end{proof}

\begin{example}\label{example: map of A_2}
 We continue from \cref{example: coho ring A_2} to describe explicitly the homomorphism
  $\phi: H^*(X_{P_\lambda/W})\to H^*(X_{P_\lambda})^W$
   when $R$ is of type $A_2$ and $M$ is the root lattice. We have
   \begin{align*}
    \ell_{E_2}-\ell_{E_0}=\ell_{H_1},\quad \ell_{E_4}-\ell_{E_0}=\ell_{H_1}+\ell_{H_2},\quad \ell_{E_3}-\ell_{E_1}=\ell_{H_1}+\ell_{H_2},\quad \ell_{E_5}-\ell_{E_1}=\ell_{H_2}.
   \end{align*}
   Then the homomorphism 
   \begin{align*}
    \phi: \frac{\Q[X_{E_0}, X_{E_1},Y_1, Y_2]}{I_S+J_S}\to \left(\frac{\Q[X_{E_{i}}: 0\leq i\leq 5]}{I_{P_\lambda}+J_{P_\lambda}}\right)^{W}
   \end{align*}
   is determined by 
  $$
    \begin{tabular}{lcl}
     $\phi(X_{E_0})=X_{E_0}+X_{E_2}+X_{E_4}$\,, & & $\phi(X_{E_1})=X_{E_1}+X_{E_3}+X_{E_5}$\,,\\
     $\phi(Y_{1})=X_{E_2}+X_{E_3}+X_{E_4}$\,, & & $\phi(Y_{2})=X_{E_3}+X_{E_4}+X_{E_5}$\,.
    \end{tabular}
  $$
   We verify \cref{lem:kernel I_K} in this case:
   \begin{align*}
     \phi(X_{E_0}Y_1)=X_{E_0}\phi(Y_1)+r_1\left(X_{E_0}\phi(Y_1)\right)+r_2r_1\left(X_{E_0}\phi(Y_1)\right)\in I_{P_\lambda},
   \end{align*}
   and similarly $\phi(X_{E_1}Y_2)\in I_{P_\lambda}$.
\end{example}

\section{Polytopal algebras}\label{section:polytopal algebra}
In this section, we continue from \cref{section:faces}. We introduce the definition and properties of the polytopal algebra $\mathcal{A}(Q)$ for a full-dimensional polytope $Q$ in $V$. We then prove \cref{general ana ring iso} when $P$ is simple and non-degenerate. The ground field is always $\R$ unless specified otherwise.

For each facet $F$ of $Q$, we choose the normal vector $\ell_F$ to have norm $1$.

\begin{definition}\label{def:polytopal algebra}
  The \textbf{(graded) polytopal algebra $\mathcal{A}(Q)$} over $\R$ is defined as 
  $$\mathcal{A}(Q)=\frac{\mathcal{SR}(Q)}{J}=\frac{\R[X_{F}: \mbox{$F$ is a facet of $Q$\,}]}{I+J},$$
  where each $X_F$ has degree 1, and  the ideal $J$ of $\mathcal{SR}(Q)$ is generated by $n$ linear forms
  \begin{align}
    \eta_i:=\sum_{F}\ip{\alpha_i,\ell_F}X_F\,,\quad\mbox 
  {where $F$ ranges over all facets of $Q$,}\label{eq:lsop}
  \end{align}
  for $\alpha_i\in S$. We write $\mathcal{A}^{i}(Q)$ for the degree-$i$ component of  $\mathcal{A}(Q)$.
\end{definition}

Note the following differences between $H^*(X_Q;\Q)$ in \cref{thm:danilov cohomology} and $\mathcal{A}(Q)$ in \cref{def:polytopal algebra}: 
the degrees of the generators $X_F$ differ due to the existence of a topological counterpart, and the ideals $J$ are defined slightly differently.
 We hope that these differences  will cause no ambiguity. In fact, the following result shows that the two algebras are isomorphic and that the polytopal algebra is ``well-defined''.

\begin{proposition}\phantom{}
  \label{prop:independent model}
  \begin{itemize}
    \item[$(1)$] The polytopal algebra $\mathcal{A}(Q)$ is defined independently from the models of the given root system.
    \item[$(2)$] Suppose that there is a lattice in $V$ to define $X_Q$. There is a natural algebra isomorphism $\mathcal{A}(Q)\cong H^{*}(X_{Q};\R)$ which  assigns each generator $X_F$ to a multiple of it. 
    \item[$(3)$] When $R$ is crystallographic, the above assignment induces isomorphisms $\mathcal{A}(P)\cong H^{*}(X_{P};\R)$ and $\mathcal{A}(P/W_K)\cong H^{*}(X_{P/W_K};\R)$.
  \end{itemize}
\end{proposition}
\begin{proof}
  When the type of $(R,S)$ is fixed, one can obtain another model $(R',S')$ of the same type by adjusting $S$; see \cite[Remark 1.2, Proposition 2.1]{humphreysReflectionGroupsCoxeter1992}. In fact, $S'$ is obtained from $S$ by scaling the simple roots or by applying an orthogonal transformation of $V$.

For $(1)$, $(2)$, and $(3)$, we obtain isomorphic algebras by adjusting the ideal $J$ (specifically, the $\eta_i$) in $\mathcal{SR}(Q)$. The argument is then analogous to that in the proof of \cref{prop:indep lattice}.
\qedhere
\end{proof}

We describe the polytopal algebras of $P$ and $P/W_K$ as follows:
\begin{itemize}
    \item The polytopal algebra $\mathcal{A}(P)$ is 
  $$ \mathcal{A}(P)=\frac{\mathcal{SR}(P)}{J_P}=\frac{\R[X_{F,s}: (F,s)\in \mathcal{F}]}{I_P+J_P},
  $$
  where  the ideal $J_P$ is generated by $n$ linear forms
  $$\sum_{(F,s)\in\mathcal{F}}\ip{\alpha_i,\ell_{s(F)}}X_{F,s}\qquad \mbox{for all}\,\, \alpha_i\in S.$$ 
  The algebra $\mathcal{A}(P)$ inherits the $W$-action on $\mathcal{SR}(P)$ described in \cref{eq:W-action}.
  \item The polytopal algebra $\mathcal{SR}(P/W_K)$ is
  $$ \mathcal{A}(P/W_K)=\frac{\mathcal{SR}(P/W_K)}{J_K}=\frac{\R[X_{F,e}, Y_k:\, F\in\mathcal{F}_K,\, k\in K]}{I_K+J_K},
  $$
  where the ideal $J_K$ is generated by $n$ linear form
  $$\sum_{F\in\mathcal{F}_K}\ip{\alpha_i,\ell_{F}}X_{F,e}+\sum_{k\in K}\ip{\alpha_i,\ell_{H_k}}Y_k\qquad \mbox{for all}\,\, \alpha_i\in S.$$
  \end{itemize}

\begin{remark}
 There are similar definitions of $\mathcal{A}(P_{\lambda})$ in  \cite[(1.2)]{stembridgePermutationRepresentationsWeyl1994} and \cite[\S4]{guiWeylGroupSymmetries2025}. The statement of \cref{prop:independent model} also holds for those definitions.
\end{remark}

\begin{lemma}\label{lem:vertex generate same}
  Suppose that $Q$ is simple.
  In $\mathcal{A}(Q)$, all square-free monomials $X_{F_{1}}{X_{F_{2}}}\cdots{X_{F_{n}}}$,  with $F_{1}\cap F_{2}\cap\cdots\cap F_{n}\ne\emptyset$, represent the same element up to a positive scalar.
\end{lemma}
\begin{proof}
  Suppose that $F_{0},F_{1},F_{2},\ldots,F_{{n-1}},F_{n}$ are distinct facets of $Q$ such that $F_{0}\cap F_{1}\cap\cdots\cap F_{{n-1}}\ne\emptyset$,  $F_{1}\cap F_{2}\cap\cdots\cap F_{{n}}\ne\emptyset$. 
  It suffices to show that $X_{F_{0}}X_{F_{1}}\cdots X_{F_{{n-1}}}= c\cdot X_{F_{1}}X_{F_{2}}\cdots X_{F_{{n}}}\in \mathcal{A}(Q)$ for some positive number $c$. Indeed, each vertex of $Q$ corresponds to a square-free monomial, and any two vertices are  are connected by edges of the form $F_{1}\cap F_{2}\cap\cdots\cap F_{n-1}$.

  Clearly, $\ell_{F_{0}}=\sum_{1\leq j\leq n}c_i\cdot\ell_{F_{j}}$ for some coefficients $c_j$. 
  We claim that \(c_n<0\); otherwise, the normal cones spanned by
$\ell_{F_{0}},\ldots,\ell_{F_{n-1}}$ and by
$\ell_{F_{1}},\ldots,\ell_{F_{n}}$
would intersect in the relative interiors of both cones.
 Choose $u$ in the real span of $S$ such that $\ip{u,\ell_{F_{n}}}=1$ and $\ip{u,\ell_{F_{j}}}=1$ for $1\leq j< n$. Then we obtain
 \begin{align*}
   {X_{F_{n}}}+c_nX_{F_{0}}+ \sum_{F}\ip{u,\ell_F}X_F\in J,
 \end{align*}
 where $F$ ranges over all other facets of $Q$.
 Since $Q$ is simple, these facets do not intersect
$F_{1}\cap\cdots\cap F_{n-1}$.
  Multiplying the above relation by
$X_{F_{1}}\cdots X_{F_{n-1}}$, we obtain
 \begin{align*}
  X_{F_{1}}{X_{F_{2}}}\cdots X_{F_{{n-1}}}{X_{F_{n}}}=-c_nX_{F_{0}}X_{F_{1}}\cdots X_{F_{{n-1}}}\in \mathcal{A}(Q).
 \end{align*}
 Since $c_n<0$, the scalar $-c_n$ is positive.
 \qedhere
\end{proof}

Let $f(Q)=(f_{-1},f_0,\ldots,f_{n-1})\in \Z^{n+1}$ denote the \textbf{$f$-vector} of the polytope $Q$, 
where $f_i$ is the number of faces of codimension $i+1$. 
The \textbf{$h$-vector} $h(Q)=(h_0,h_1,\ldots,h_n)$ and the \textbf{$h$-polynomial} $h_Q(t)$  are defined by 
\begin{align*}
  h_Q(t)=\sum_{i=0}^nh_it^i=\sum_{i=0}^nf_{i-1}t^i(1-t)^{n-i}.
\end{align*}
Note that the $\R$-basis of $\mathcal{SR}(Q)$ consists of all monomials $X_{F_{1}}^{n_1}X_{F_{2}}^{n_2}\cdots X_{F_{p}}^{n_p}$ with $F_{1}\cap F_{2}\cap\cdots\cap F_{p}\ne\emptyset$. 
Thus, the Hilbert series of  $\mathcal{SR}(Q)$ is 
\begin{align*}
  \sum_{F}\frac{t^{\mathrm{codim}(F)}}{(1-t)^{\mathrm{codim}(F)}}=\sum_{i=0}^nf_{i-1}\frac{t^i}{(1-t)^{i}}=\frac{h_Q(t)}{(1-t)^n},
\end{align*}
where $F$ ranges over all faces of $Q$.

\begin{lemma}\label{lem:isPDA}
 If $Q$ is simple, then $\mathcal{A}(Q)$ is a Poincar\'{e}  duality algebra over $\R$.
\end{lemma}
\begin{proof}
  The $n$ linear forms $\eta_i$ in \cref{eq:lsop} are a \textbf{linear system of parameters} of $\mathcal{SR}(Q)$ (\cite[remark (p.150)]{stanleyBalancedCohenmacaulayComplexes1979}). 
  Since $\mathcal{SR}(Q)$ is \textbf{Cohen-Macaulay}, the $\eta_i$ for $1\leq i\leq n$ are algebraically independent (\cite[Corollary II.4.4]{stanleyCombinatoricsCommutativeAlgebra1996}
  ). Thus, there is an $\R$-module isomorphism (\cite[Corollary I.5.10]{stanleyCombinatoricsCommutativeAlgebra1996})
  $$
  \mathcal{SR}(Q)\cong \mathcal{A}(Q)\otimes\R[\theta_i:1\leq i\leq n].
  $$
  Consequently, $h_Q(t)$ is the Hilbert series of $\mathcal{A}(Q)$.
By \cite[Theorem II.1.3, Corollary II.5.2, p. 67]{stanleyCombinatoricsCommutativeAlgebra1996}, 
the $h$-vector $h(Q)$ satisfies $h_i=h_{n-i}$ for all $1\leq i\leq n$. In particular, $h_0=h_n=1$.
Finally,  by \cite[Remark 2.11.4, Theorem 2.79]{harimaLefschetzProperties2013}, $\mathcal{A}(Q)$ is a  PDA.
\qedhere
\end{proof}

\begin{proof}[\textbf{Proof of \cref{general ana ring iso}}]\label{pf:general ana iso}
The formula in \cref{eq:coefficients of s} still holds, now with all coefficients $C_{F,s,k}\in\R$.
We define the homomorphism
\begin{align}
    \psi:\mathcal{A}(P/W_K)&\to\mathcal{A}(P)^{W_K}\label{eq:general iso}
  \end{align}
  $$
   \psi(X_{F,e})=\sum_{s\in{W_K}/{W_F}}X_{F,s} \quad\text{and}\quad
  \psi(Y_{k})=\sum_{(F,s)\in\mathcal{F}} C_{F,s,k}X_{F,s}\,.
  $$
  Following the steps in \cref{section:proofs coho}, one can verify that
$\psi$ is well-defined and surjective.
The remaining step is to  show that $\mathcal{A}(P)^{W_K}\ne 0$; then \cref{lem:sur is iso} implies that $\psi$ is an isomorphism.
  
   By \cite[Lemma 10.7.1]{danilovGEOMETRYTORICVARIETIES1978}, $\mathcal{A}^n(P)=\R$ is generated by square-free monomials. 
   Suppose that $P$ has facets $F_1,F_2\ldots,F_n$  intersecting at a vertex, then $X_{F_1}X_{F_2}\cdots X_{F_n}\ne 0\in\mathcal{A}(P)$. 
   By \cref{lem:vertex generate same}, the sum $\sum_{s\in W_K}s(X_{F_1}X_{F_2}\cdots X_{F_n})$ is nonzero in $\mathcal{A}(P)^{W_K}$.
   \qedhere
\end{proof}

\section{Further Questions}\label{section:generalization}

In this section, we continue to work under the conditions stated at the beginning of \cref{section:faces}.
We discuss some interesting properties and open questions concerning $W$-symmetric polytopes $P$ and the associated polytopal algebras $\mathcal{A}(P)$.
In particular, we focus on the isomorphism between $\mathcal{A}(P/W_K)$ and $\mathcal{A}(P)^{W_K}$, and on the $W$-representation $\mathcal{A}(P)$.

\subsection{On isomorphisms}
It is not straightforward to generalize \cref{general ana ring iso} to the degenerate simple $W$-symmetric polytopes $P$, because the face structures of
 $P$ and $P/W_K$ become more complicated.

When the dimension $n$ is 2, the $W$-symmetric polygon $P$ has a relatively easy face structure. Song \cite{songToricSurfacesReflection2022} classified $W$-symmetric polygons to study the corresponding toric varieties, when they exist.
More explicitly, a $W$-symmetric polygon $P_{\Lambda}$ is classified by the size of $\Lambda\cap \partial C_K$, which can be 0, 1, or 2. 
This classification determines and is determined by three possible shapes of $P_{\Lambda}/W_K$, shown in \cref{figure:three shapes}.
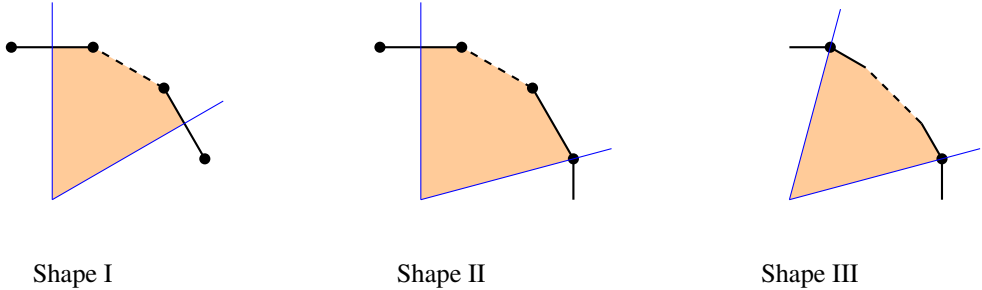
\begin{figure}[htpb]
  \centering
  \begin{tikzpicture}
    \begin{scope}[scale=0.7,shift={(-7,0)}]
    \newdimen\r
    \r=3cm

    \fill[orange!50,opacity=0.8] (0:0)-- (30:{cos(15)*\r})--(45:\r)--(75:\r)--(90:{cos(15)*\r})--cycle;

    \draw[thick] (75:\r)--(105:\r);
    \draw[thick] (15:\r)-- (45:\r);

    \draw[thick,dashed] (45:\r)-- (75:\r);
    \fill (75:\r) circle (3pt);
    \fill (105:\r) circle (3pt);
    \fill (15:\r) circle (3pt);
    \fill (45:\r) circle (3pt);

    \draw[blue] (0:0) -- (30:{1.25*\r});
    \draw[blue] (0:0) -- (90:{1.25*\r});

    \node at (285:{0.5*\r}) {Shape I};

  \end{scope}

  \begin{scope}[scale=0.7]%0
    \newdimen\r
    \r=3cm

    \fill[orange!50,opacity=0.8] (0:0)--(15:\r) --(45:\r)--(75:\r)--(90:{cos(15)*\r})--cycle;
    
    \draw[thick] (75:\r)--(105:\r);
    \draw[thick] (15:\r)-- (45:\r);
    \draw[thick,dashed] (45:\r)-- (75:\r);
    \draw[thick] (15:\r)-- (0:{cos(15)*\r});
    \fill (75:\r) circle (3pt);
    \fill (105:\r) circle (3pt);
    \fill (15:\r) circle (3pt);
    \fill (45:\r) circle (3pt);

    \draw[blue] (0:0) -- (15:{1.25*\r});
    \draw[blue] (0:0) -- (90:{1.25*\r});

    \node at (285:{0.5*\r}) {Shape II};

  \end{scope}

  \begin{scope}[scale=0.7,shift={(7,0)}] %1
    \newdimen\r
    \r=3cm
    
    \fill[orange!50,opacity=0.8] (0:0)--(15:\r) --(30:{cos(15)*\r}) -- (60:{cos(15)*\r})--(75:\r)--cycle;

    \draw[thick] (75:\r)--(90:{cos(15)*\r});
    \draw[thick] (15:\r)-- (30:{cos(15)*\r});
    \draw[thick,dashed] (30:{cos(15)*\r}) -- (60:{cos(15)*\r});

    \draw[thick] (60:{cos(15)*\r})-- (75:\r);
    \draw[thick] (15:\r)-- (0:{cos(15)*\r});

    \fill (75:\r) circle (3pt);
    
    \fill (15:\r) circle (3pt);

    \draw[blue] (0:0) -- (15:{1.25*\r});
    \draw[blue] (0:0) -- (75:{1.25*\r});

    \node at (285:{0.5*\r}) {Shape III};
    \end{scope}

 \end{tikzpicture}
 \caption{Three possible shapes of $P_{\Lambda}$. The orange area indicates $P_{\Lambda}\cap C_K$. When $|K|=1$, the blue polyline is actually a straight line.}\label{figure:three shapes}
\end{figure}

\begin{theorem}\label{generalization to polygons}
  For a $W$-symmetric polygon $P_{\Lambda}$ and any parabolic subgroup $W_K$,
  there is an explicit ring isomorphism
  $$\mathcal{A}(P_{\Lambda}/W_K)\cong \mathcal{A}(P_{\Lambda})^{W_K}.$$
\end{theorem}
\begin{proof}
   The map $\psi: \mathcal{A}(P/W_K)\to \mathcal{A}(P)^{W_K}$ is constructed as in \cref{eq:general iso}.  
   Repeating the arguments in \cref{section:proofs coho} and \cref{section:polytopal algebra}, we have that:
    the range of $\psi$ is well-defined; $\psi$ maps $J_K$ of $\mathcal{SR}(P/W_K)$ to 0; $\mathcal{A}(P/W_K)$ and $\mathcal{A}(P)$ are both Poincar\'{e}  duality algebras with top degree 2; and $(\mathcal{A}^2(P))^{W_K}=\R$.

    We now verify that \cref{lem:kernel I_K} holds, i.e., that $\psi$ maps $I_K$ of $\mathcal{SR}(P/W_K)$ to 0.
    It suffices to show that $\psi(X_{F,e}Y_{1})=0$ when $F\in \mathcal{F}_K$ and $H_1$ are disjoint and support edges of $P/W_K$. 
    One has
   \begin{align*}
    \psi\left(X_{F,e}Y_{1}\right)=\sum_{s\in {W_K}/{W_{F}}}s\left(\sum_{(F',t)\in\mathcal{F}}  C_{F',t,1}X_{F,e}X_{F',t}\right). 
  \end{align*} 
  The only possible case for $X_{F,e}X_{F',t}\ne 0$ with $t(F')\ne F'$ is when $F$ intersects $H_2$ (if $H_2$ supports $\partial C_K$). Then either $r_2(F)=t(F')$ or $r_2(F)=F$. Hence $t=r_2$ and $C_{F',t,1}=0$ whenever $X_{F,e}X_{F',t}\ne 0$.

  To complete the proof, it remains to show that $\psi$ is surjective; then \cref{lem:sur is iso} implies that $\psi$ is an isomorphism. Since $\mathcal{A}^0(P)=\mathcal{A}^2(P)=\R$, it suffices to show that  $(\mathcal{A}^1(P))^{W_K}$ is generated by the elements $\sum_{s\in W_K/W_F}X_{F,s}$ for $F\in\mathcal{F}_K$. Clearly, these elements generate $(\mathcal{SR}^1(P))^{W_K}$, and hence also $(\mathcal{A}^1(P))^{W_K}$.
\qedhere
\end{proof}

So we would like to ask:
\begin{question}[{cf. \cite[Question 8.2]{horiguchiToricOrbifoldsAssociated2024}}]
  For which $W$-symmetric polytopes does the isomorphism in \cref{general ana ring iso} hold?
\end{question}

We are closer to understanding the 3-dimensional case, as part of \cref{lem:simple flag} holds as follows:
\begin{lemma}
  If a 3-dimensional $W$-symmetric polytope $P$ is simple, then so is the quotient polytope $P/W_K$.
\end{lemma}
\begin{proof}
  Suppose that a vertex of $P/W_K$ is contained in at most the following facets: $F_1,\ldots,F_p \in \mathcal{F}_K$ ($1 \le p \le 3$) and $H_{k_1},\ldots,H_{k_q}$ with $k_i \in K$ ($1 \le q \le 3$, if possible). We rule out all cases below to complete the proof.

  The case $q = 3$ is impossible, since $H_1 \cap H_2 \cap H_3 = \{0\}$ and the facets $F_i$ do not contain the origin.

  The case $p = 3$ and $q \ge 1$ is impossible. Suppose that $F_1, F_2, F_3$, and $H_1$ meet at a vertex. Since $\ip{\ell_{H_1}, \ell_{F_i}} \ne 0$ for some $i$, we may assume that $r_1(F_1) \ne F_1$. Because $P$ is simple, we have that $r_1(F_1)$ must be either $F_2$ or $F_3$, which implies that $F_2$ or $F_3$ does not lie in $\mathcal{F}_K$.

  The remaining case is $p = q = 2$. Since $\ip{\ell_{H_i}, \ell_{F_j}} \ne 0$ for some $i, j$, we may assume that $r_1(F_1) \ne F_1$. Then $r_1(F_1)$, $F_1$, and $F_2$ meet.
We then have $r_2(F_1) = F_1$ and $r_1(F_2) = r_2(F_2) = F_2$. Hence $H_1 \cap H_2 \cap F_2$ is the barycenter of $F_2$, which cannot lie in $F_1$.
   \qedhere
\end{proof}

\subsection{On representations} The polytopal algebra $\mathcal{A}(P)$ is naturally a $W$-representation. Recall that a \textbf{permutation representation} of $W$ (over $\R$) is a vector space with a basis being  a $W$-set. Consequently, the character of a permutation representation takes integer values.

\begin{example}
  We follow \cref{example: I(5) and A_2} to examine the character of $\mathcal{A}(P_{\lambda})$ as a graded $W$-representation. 
  It suffices to compute the traces of $r_1$ and $r_2$ on $\mathcal{A}^1(P_{\lambda})$.
  \begin{itemize}
      \item If $R$ is of type $I_2(5)$ with $\lambda\in H_2$, then $\mathrm{Tr}(r_1)=1$ and $\mathrm{Tr}(r_2)=1-\sqrt{5}$.
      \item If $R$ is of type $A_2$ with $\lambda\in C_S\setminus\partial C_S$, then $\mathrm{Tr}(r_1)=\mathrm{Tr}(r_2)=1$. This  also holds for  $R$  of type $I_2(m)$.
    \end{itemize}
\end{example}

Using \cref{prop:independent model}, we see that $\mathcal{A}(P_{\lambda})$ is a permutation representation when $P_{\lambda}$ is non-degenerate and $R$ is crystallographic. Indeed, \cite{stembridgePermutationRepresentationsWeyl1994} proves this for $H^*(X_{P_{\lambda}})$.
This leads to the following natural question:
\begin{question}[{cf.~\cite[Question 11.1]{stembridgePermutationRepresentationsWeyl1994}}]
For which $W$-symmetric polytopes $P$ is the polytopal algebra $\mathcal{A}(P)$ a permutation representation of $W$?
\end{question}

\printbibliography

\end{document}